\documentclass[11pt]{article}
\usepackage{amsfonts}
\usepackage{amssymb}
\usepackage{graphicx}
\usepackage{mathtools}
\usepackage{color}

\newtheorem{theorem}{Theorem}[section]
\newtheorem{proposition}[theorem]{Proposition}

\newtheorem{remark}[theorem]{Remark}
\newtheorem{definition}[theorem]{Definition}

\newtheorem{example}[theorem]{Example}

\newcommand{\graph} {\mbox{\rm graph}}

\newcommand{\Img}{\mbox{\rm im\,}}

\newcommand{\bma}{\begin{bmatrix}}
\newcommand{\ema}{\end{bmatrix}}

\newcommand{\mR}{\mathbb{R}}

\newcommand{\X}{\mathcal{X}}
\newcommand{\F}{\mathcal{F}}
\newcommand{\M}{\mathcal{M}}
\newcommand{\cR}{\mathcal{R}}
\newcommand{\D}{\mathcal{D}}
\newcommand{\cL}{\mathcal{L}}

\newcommand{\eq}{\end{equation}}
\DeclareMathOperator{\im}{im}
\DeclareMathOperator{\rank}{rank}

\def\BibTeX{{\rm B\kern-.05em{\sc i\kern-.025em b}\kern-.08em
    T\kern-.1667em\lower.7ex\hbox{E}\kern-.125emX}}

\newcommand {\proof}{\par\noindent{\it Proof}. \ignorespaces}
\def\eproof{\space
       {\ \vbox{\hrule\hbox{\vrule height1.3ex\hskip0.8ex\vrule}\hrule}}
       \bigskip}
\definecolor{myBlue}{RGB}{30,144,255} 
\definecolor{myGreen}{RGB}{69,169,0} 
\definecolor{myRed}{RGB}{165,12,42}
\definecolor{myOrange}{RGB}{225,92,22}
\definecolor{color0}{rgb}{0.121568,0.46666,0.70588}
\definecolor{color1}{rgb}{1,0.4980,0.0549}
\definecolor{color2}{rgb}{0.17254,0.62745,0.17254901}
\definecolor{color3}{rgb}{0.83921,0.15294,0.156862}
\definecolor{color4}{rgb}{0.580392,0.4039215,0.7411764}
\definecolor{color5}{rgb}{0,0,0}

\newcommand{\delete}[1]{ }

\def\C{\mathbb{C}}

\makeatletter
\newsavebox{\@brx}
\newcommand{\llangle}[1][]{\savebox{\@brx}{\(\m@th{#1\langle}\)}%
  \mathopen{\copy\@brx\kern-0.5\wd\@brx\usebox{\@brx}}}
\newcommand{\rrangle}[1][]{\savebox{\@brx}{\(\m@th{#1\rangle}\)}%
  \mathclose{\copy\@brx\kern-0.5\wd\@brx\usebox{\@brx}}}
\makeatother

\newcommand{\E}{\mathcal{E}}

\newcommand {\mycomment}[1]{} 
\newcommand {\mat}  [1] {\left[\begin{array}{#1}}
\newcommand {\rix}      {\end{array}\right]}

\catcode`@=11     
\font\tenex=cmex10 
\newdimen\p@renwd
\setbox0=\hbox{\tenex B} \p@renwd=\wd0 
\def\bmat#1{\begingroup \m@th
  \setbox\z@\vbox{\def\cr{\crcr\noalign{\kern2\p@\global\let\cr\endline}}%
    \ialign{$##$\hfil\kern2\p@\kern\p@renwd&\thinspace\hfil$##$\hfil
      &&\quad\hfil$##$\hfil\crcr
      \omit\strut\hfil\crcr\noalign{\kern-\baselineskip}%
      #1\crcr\omit\strut\cr}}%
  \setbox\tw@\vbox{\unvcopy\z@\global\setbox\@ne\lastbox}%
  \setbox\tw@\hbox{\unhbox\@ne\unskip\global\setbox\@ne\lastbox}%
  \setbox\tw@\hbox{$\kern\wd\@ne\kern-\p@renwd\left[\kern-\wd\@ne
    \global\setbox\@ne\vbox{\box\@ne\kern2\p@}%
    \vcenter{\kern-\ht\@ne\unvbox\z@\kern-\baselineskip}\,\right]$}%
  \null\;\vbox{\kern\ht\@ne\box\tw@}\endgroup}
\catcode`@=12    
\newif\ifMDlatex
\catcode`@=11     
\def\MD@us#1{\csname#1style\endcsname}
\def\MD@uf#1{\csname#1font\endcsname}
\def\MD@t{text}
\def\MD@s{script}
\def\MD@ss{scriptscript}
\newdimen\MD@unit
\MD@unit\p@
\def\MD@changestyle#1{
  \relax\MD@unit0.1\fontdimen6\MD@uf{#1}0
  \everymath\expandafter{\the\everymath\MD@us{#1}}
}
\def\MD@dot{$\m@th\ldotp$}
\def\MD@palette#1{\mathchoice{#1\MD@t}{#1\MD@t}{#1\MD@s}{#1\MD@ss}}
\def\MD@ddots#1{{\MD@changestyle{#1}%
  \mkern1mu\raise7\MD@unit\vbox{\kern7\MD@unit\hbox{\MD@dot}}%
  \mkern2mu\raise4\MD@unit\hbox{\MD@dot}%
  \mkern2mu\raise \MD@unit\hbox{\MD@dot}\mkern1mu}}%
\def\MD@iddots#1{{\MD@changestyle{#1}%
  \mkern1mu\raise \MD@unit\hbox{\MD@dot}%
  \mkern2mu\raise4\MD@unit\hbox{\MD@dot}%
  \mkern2mu\raise7\MD@unit\vbox{\kern7\MD@unit\hbox{\MD@dot}}}}%
\def\MD@vdots#1{\vbox{\MD@changestyle{#1}%
    \baselineskip4\MD@unit\lineskiplimit\z@
    \kern6\MD@unit\hbox{\MD@dot}\hbox{\MD@dot}\hbox{\MD@dot}}}%
\ifMDlatex
  \DeclareRobustCommand\ddots{\mathinner{\MD@palette\MD@ddots}}%
  \DeclareRobustCommand\iddots{\mathinner{\MD@palette\MD@iddots}}%
  \DeclareRobustCommand\vdots{\mathinner{\MD@palette\MD@vdots}}%
\else
  \def\ddots{\mathinner{\MD@palette\MD@ddots}}%
  \def\iddots{\mathinner{\MD@palette\MD@iddots}}%
  \def\vdots{\mathinner{\MD@palette\MD@vdots}}%
\fi
\catcode`@=12    

\begin{document}

\title
{Nonlinear differential-algebraic systems with dissipative port-Hamiltonian structure}
\author{
\and V. Mehrmann\footnotemark[3]~\footnotemark[1]
\and  A.J. van der Schaft \footnotemark[2]
}
\maketitle
\begin{abstract}
Different representations  of nonlinear  dissipative Hamiltonian and port-Hamiltonian (pH) differential-algebraic equations (DAE) systems are derived and compared.
Using global geometric and algebraic points of view, including singular Lagrange and Dirac structures, as well as Morse families and their skew counterparts, constraints are incorporated in the models. The results are illustrated specifically with constrained mechanical systems using kinematic and geometric constraints.  The relation between the different representations is studied and in particular the conversion of Lagrange into Dirac algebraic constraints and vice versa is discussed.
To incorporate dissipation also maximally monotone structures are studied.
\end{abstract}
\noindent

\noindent


\renewcommand{\thefootnote}{\fnsymbol{footnote}}
\footnotetext[1]{
Institut f\"ur Mathematik, IMoS/MA 4-5, TU Berlin, Fasanenstr. 89,
10623 Berlin, Germany.
\texttt{mehrmann@math.tu-berlin.de}.
}
\footnotetext[2]{
Bernoulli institute for Mathematics, Computer Science and Artificial Intelligence, Jan C. Willems Center for Systems and Control, University of Groningen, Nijenborgh 9, Groningen, the Netherlands. \texttt{a.j.van.der.schaft@rug.nl}.
}
\renewcommand{\thefootnote}{\arabic{footnote}}

\section{Introduction}\label{sec:intro}
In recent years the energy based modeling concept of port-Hamiltonian (pH) systems has become a very successful paradigm that is used in all areas of science and engineering, see e.g. \cite{SchJ14}. To incorporate constraints in pH systems such as those that arise in electrical or transport networks via Kirchhoff's laws, see e.g. \cite{GueF99a,GueF99b}, as balance equations in chemical engineering, see e.g. \cite{DieLSBS2002,DieUFSABBGKSS01}, or as geometric and/or kinematic (holonomic or non-holonomic) constraints in mechanical (multi-body) systems, see e.g. \cite{EicF98,RabR02,Sch90}, in recent years the concept of pH systems has been extended in various formulations to pH differential-algebraic equations (pHDAEs), see \cite{AltS25,BeaMXZ18,MehM19,MehU23,Sch13}. The most general form of so called \emph{extended pHDAEs} was introduced  in \cite{SchM18,SchM20} and analyzed in detail in \cite{MehS23,SchM23}  from an algebraic and geometric point of view for the case of linear systems.

In this paper we continue this analysis for nonlinear extended pHDAE systems, and in particular, we present a systematic study of the incorporation of constraints in the system via Lagrange multipliers using global geometric and algebraic points of view, including singular Lagrange and Dirac structures, as well as Morse families and their skew counterparts, see \cite{DalS99,SchM18}. The results are illustrated  with constrained mechanical systems using kinematic and geometric constraints.  The relation between the different representations is studied and in particular the conversion of Lagrange into Dirac algebraic constraints and vice versa is discussed. This resolves, in particular, an often discussed question of the energy based (pH) formulation of mechanical systems with constraints.

To incorporate dissipation into the system, then the concepts are extended by using nonlinear maximally monotone structures and appropriate representations of ports.

We discuss only real systems although most of the results can be formulated for complex problems in the same way. We also restrict ourselves to \emph{regular} DAEs,  see \cite{KunM24} for a detailed discussion of existence and uniqueness of solutions (in different function spaces) of more general singular DAEs.

The paper is organized as follows. In Section~\ref{sec:prelim} we introduce the necessary notation and in Section~\ref{sec:phdaes} the concepts of linear and nonlinear port-Hamiltonian differential-algebraic equations without dissipation and ports are discussed.
In Section~\ref{sec:augmented}
we show how to augment port-Hamiltonian systems by adding the constraints via Lagrange multipliers and we show how Lagrange algebraic constraints can be converted into Dirac algebraic constraints, and conversely.
As examples, we use constrained mechanical systems as well as optimality systems in optimal control.
In Section~\ref{sec:nlmech} we specifically discuss the pH formulation of nonlinear mechanical systems with geometric and kinematic constraints including the widely unknown Livens principle. Section~\ref{sec:nldiss} then uses the concept of maximally monotone structures to extend the previous results to systems with energy dissipation and ports.
The paper finishes with some conclusions and further research questions in Section~\ref{sec:conclusion}.

\section{Notation}\label{sec:prelim}
Consider an n-dimensional  linear space $\mathcal X$ with elements denoted by $x$. Let $T_x\mathcal X$ denote the \emph{tangent space} to $\mathcal X$ at $x \in \mathcal X$. Obviously, since $\mathcal X$ is a linear space $T_x\mathcal X$ can be identified with $\mathcal X$ for any $x$. Nevertheless, the distinction will be important if we generalize (later on) to \emph{manifolds} $\X$ (instead of linear spaces). The elements of $T_x{\mathcal X}$ will be denoted by $f$ and will be called \emph{flow} vectors. Furthermore, consider the \emph{dual space} $T_x^*{\mathcal X}$ of $T_x{\mathcal X}$, with elements denoted by $e$ and called \emph{effort} vectors. Obviously, in the linear case $T_x^*{\mathcal X}$ can be identified with $\X^*$.

The space $\mathcal X \times \mathcal X^*$ is endowed with two natural bilinear forms. First, the \emph{plus}-bilinear form
\begin{equation}\label{bilind}
\begin{array}{l}
\langle (f_1,e_1), (f_2,e_2) \rangle_+ := <e_2,f_1> + <e_1,f_2>,
\end{array}
\end{equation}
where $\langle \cdot ,\cdot \rangle$ denotes the duality product between the dual spaces $\mathcal X$ and $\mathcal X^*$.
Second, the \emph{minus}-bilinear form
\begin{equation}\label{bilinl}
\begin{array}{l}
\langle (x_1,e_1), (x_2,e_2) \rangle_- := <e_2,x_1> - <e_1,x_2>.
\end{array}
\end{equation}
In coordinates for $\mathcal X$ and dual coordinates for $\mathcal X^*$ these bilinear forms are represented by the  matrices
\[
\Sigma_+ := \mat{cc} 0 & I_n \\ I_n & 0 \rix, \quad \Sigma_- :=\mat{cc} 0 & I_n \\ -I_n & 0 \rix,
\]
respectively.
\begin{definition}\label{def:dirlag}
A subspace $\mathcal D \subset \mathcal X \times \mathcal X^*$ is called a \emph{Dirac structure} if the bilinear form $\langle \cdot, \cdot \rangle_+$ is zero on $\mathcal D$ and, moreover, $\mathcal D$ is maximal with respect to this property. A subspace $\mathcal L \subset \mathcal X \times \mathcal X^*$ is called a \emph{Lagrange structure}\footnote{Normally called a Lagrangian subspace. However, we adopt the terminology Lagrange structure in order to stress the similarity with Dirac structures.}
if the bilinear form $\langle \cdot, \cdot \rangle_-$ is zero on $\mathcal L$ and, moreover, $\mathcal L$ is maximal with respect to this property. Furthermore, a Lagrange structure $\mathcal L$ is called \emph{nonnegative} if the bilinear form $\Sigma_+$ is nonnegative whenever restricted to $\mathcal L$.
\end{definition}

Note that for any $(f,e) \in \D$ we have $0=\langle (f,e), (f,e) \rangle_+ = 2 \langle e,f\rangle$. This  corresponds physically to \emph{power conservation}. In fact, an equivalent definition of a Dirac structure $\D$ is that $\langle e,f\rangle =0$ for any $(f,e) \in \D$, and that $\D$ is maximal with respect to this property; cf. \cite{SchJ14}.

A special type of Dirac structure is the graph of any \emph{skew-symmetric} linear map from $\X^*$ to $\X$, see also Appendix A.
Similarly, a special type of Lagrange structure is the graph of a \emph{symmetric} linear map from $\X$ to $\X^*$. This last fact means that any quadratic function $\mathcal H(x)=\frac{1}{2}x^\top Qx$ on $\X$ with $Q=Q^\top$ defines a Lagrange structure. Indeed the graph of the map defined by the gradient of this quadratic function, i.e.,
\begin{equation*}
\{(x,e) \in \X \times \X^* \mid e= \frac{\partial \mathcal H}{\partial x}(x)=Qx \}
\end{equation*}
is a Lagrange structure.

A Dirac structure $\mathcal D \subset \mathcal X \times \mathcal X^*$ is called \emph{regular} if its projection on $\mathcal X^*$ is equal to $\X^*$, or equivalently, if $\D$ is the graph of a linear map from $\X^*$ to $\X$. If not, then we will call $\D$ a \emph{singular Dirac structure}. Similarly, we will call a Lagrange structure $\mathcal L \subset \mathcal X \times \mathcal X^*$ \emph{regular} if the projection of $\mathcal L$ on $\mathcal X$ is equal to $\X$, or equivalently, if $\cL$ is the graph of a linear map from $\X$ to $\X^*$. If this is not the case, then $\cL$ is called \emph{singular Lagrange structure}.

Dirac and Lagrange structures admit the following coordinate representations, see e.g. \cite{SchJ14,SchM18}. Take any coordinates $x$ for the $n$-dimensional linear space $\X$ and dual coordinates for $\X^*$. Then using matrices $K,L \in \mathbb R^{n,n}$, any \emph{Dirac structure} $\mathcal D \subset \mathcal X \times \mathcal X^*$ admits the kernel/image representation
\begin{equation}\label{diracspace}
\mathcal D = \ker \mat{cc} K & L \rix = \Img \mat{cc} L^\top \\ K^\top \rix \subset {\mathcal X} \times \mathcal X^*,
\end{equation}
with $K,L$ satisfying $\rank \mat{cc} K & L\rix=n$ and the generalized \emph{skew-symmetry condition}
\begin{equation}\label{dirac}
 KL^\top = - L K^\top.
\end{equation}
Conversely, any such pair $K,L$ defines a Dirac structure.

Analogously, any Lagrange structure $\mathcal L \subset {\mathcal X} \times \mathcal X^*$ can be represented as
\begin{equation}\label{lagrangespace}
\mathcal L = \ker \mat{cc} -S^\top & P^\top \rix = \Img \mat{cc} P \\ S \rix \subset {\mathcal X} \times \mathcal X^*,
\end{equation}
for certain matrices $S,P  \in \mathbb R^{n,n}$ satisfying $\rank \mat{c} P \\ S \rix =n$, as well as the generalized \emph{symmetry condition}
\begin{equation} \label{lagrange}
S^\top P= P^\top S.
\end{equation}
Conversely any such pair $P,S$ defines a Lagrange structure.

The above definition of  Dirac structures  and Lagrange structures modeled on linear spaces $\X$ can be extended to manifolds as follows.
Let $\mathcal X$ be a finite-dimensional \emph{manifold}. Denote by $T_x{\mathcal X}$ the \emph{tangent space} to $\mathcal X$ at $x \in \mathcal X$,
and by the dual space $T_x^*\mathcal X$ the \emph{cotangent space} to $\mathcal X$ at $x$. Then for every $x \in \X$, the product space $T_x{\mathcal X} \times T^*_x{\mathcal X}$ can be endowed with the bilinear forms defined in \eqref{bilind} and \eqref{bilinl}.

A \emph{Dirac structure} modelled on a \emph{manifold} $\mathcal X$ is defined as a
subbundle\footnote{$T\mathcal X \oplus  T^*\X$ is the vector bundle over $\X$ with fiber $T_x \X \times T^*_x \X$ at any $x \in \X$.} $\mathcal D \subset T\mathcal X \oplus  T^*\X$ such that for every $x \in \mathcal X$ the subspace $\D(x)\subset  T_x \X \times T^*_x \X$ is a Dirac structure as in Definition~\ref{def:dirlag}. A Dirac structure $\mathcal D$ on $\mathcal X$ is called \emph{regular} if for any $d \in \D \subset T\mathcal X \oplus  T^*\X$ the subspace $\mathcal D(x) \subset T_x\X \times T_x^*\X$, where $x$ is the projection of $d$ on $\X$, can be written as the graph of a linear mapping $J(x): T_x^* \X \to T_x\X$. In this case, $J(x)$ is skew-symmetric. Otherwise, the Dirac structure is called \emph{singular}.

Furthermore, a submanifold $\cL \subset T^*\X$ is called \emph{Lagrangian}\footnote{Usually a Lagrangian submanifold is equivalently defined as a maximal submanifold restricted to which the canonical symplectic form on $T^*\X$ is zero.} if at every $\ell \in \cL$ the tangent space $T_\ell \cL$ is a Lagrangian subspace of $T_x \X \times T_x^* \X$, where $x$ is the projection of $\ell \in T^*\X$ on $\X$.
The Lagrangian submanifold $\cL$ is called \emph{regular} if $\cL$ locally can be written as the graph of the gradient of a differentiable function $\mathcal H: \X \to \mathbb{R}$, that is,
\begin{equation*}
\cL:=\{(x, e) \in T^*\X \mid e = \frac{\partial \mathcal H}{\partial x}(x) \}
\end{equation*}
(and conversely, for any differentiable $\mathcal H$ this defines a Lagrangian submanifold). Equivalently, $\cL$ is regular if locally the projection of $\cL \subset T^*\X$ on $\X$ is equal to $\X$. If there does not exist locally such a function $\mathcal H$, then the Lagrangian submanifold $\cL$ will be called \emph{singular}.

\section{Port-Hamiltonian DAE systems without dissipation and external ports}\label{sec:phdaes}
In this section the definition of  port-Hamiltonian DAE systems without energy dissipation and without external ports will be recalled.  The extension to systems with energy dissipation and external ports will be given in Section \ref{sec:nldiss}.

\subsection{Linear port-Hamiltonian DAE systems without dissipation}
Let us start with the linear case following \cite{SchM18}.
\begin{definition}\label{def:defph}
A \emph{linear port-Hamiltonian DAE system} (pHDAE) on a linear state space $\X$ is defined by a pair $(\D,\cL)$, where $\D \subset \X\times \X^*$ is a Dirac structure, and $\cL \subset {\X} \times \X^*$ is a Lagrange structure. Its dynamics is given by the relation
\begin{equation}\label{phdef}
\{(x,\dot{x}) \mid \exists\, e \in \X^* \mbox{ s.t. } (x,e) \in \cL, (-\dot{x},e) \in \D \}.
\end{equation}
\end{definition}
In coordinates $x$ for $\X$ the Dirac structure $\mathcal D$ is specified as before by a pair of $n \times n$ matrices $K,L$, and the Lagrange structure $\cL$ is determined by a pair of $n \times n$ matrices $P,S$. Then in view of \eqref{lagrangespace} $(x,e) \in \cL$ can be written as $x=Pz, e=Sz$ for a certain $n$-vector $z$. Furthermore, $(-\dot{x},e) \in \mathcal D$ amounts to $K\dot{x}=Le$. Substitution of the first expressions into the second results in the following coordinate representation of a linear port-Hamiltonian DAE system
\begin{equation}
\label{kpls}
KP\dot{z} =LSz.
\end{equation}
This constitutes a system of \emph{differential-algebraic equations} (DAEs) $E\dot{z}=Az$, where $E=KP$ and $A=LS$. In fact, there are two types of algebraic constraints corresponding to singularity of $E=KP$. The first is specified by the singularity of the matrix $P$ (corresponding to singularity of the Lagrange structure $\cL$), and the second by the singularity of the matrix $K$ (corresponding to singularity of the Dirac structure $\D$). Following \cite{SchM18}, the first type of algebraic constraints are thus called \emph{Lagrange algebraic constraints}, and the second \emph{Dirac algebraic constraints}. Explicitly, the Lagrange algebraic constraints are given as $x=Pz \in \Img P$, and the Dirac algebraic constraints as $Sz \in \Img K^\top$. (Note that because of \eqref{diracspace} $e=Sz \in \Img K^\top$.)

Note that if $\cL= \{ (x,e) \mid e=Qx \}$ for some symmetric matrix $Q$, then the port-Hamiltonian DAE system $(\D,\cL)$ satisfies
\begin{equation}
\label{Qcons}
\frac{d}{dt} \frac{1}{2} x^\top Q x= x^\top Q \dot{x} =  0,
\end{equation}
since $(f=-\dot{x}, e=Qx) \in \D$. In physical applications $\frac{1}{2} x^\top Q x$ equals the total stored energy (Hamiltonian), in which case \eqref{Qcons} represents conservation of total energy.

In the case that the Lagrange structure $\cL$ is \emph{singular}, it is \emph{not} generated by a Hamiltonian $\mathcal H= \frac{1}{2} x^\top Q x$ for some $Q=Q^\top$. However in this case one may consider the quadratic function
$\frac{1}{2} z^\top S^\top P z$ of the vector $z$. Indeed,
\begin{equation*}
\frac{d}{dt} \frac{1}{2} z^\top S^\top P z = e^\top \dot{x}=0,
\end{equation*}
since $x=Pz$ and $e=Sz$, and $(-\dot{x},e)=0$. Thus $\frac{1}{2} z^\top S^\top P z$ takes over the role of the stored energy.

A special case of \eqref{kpls} are pHDAEs as defined in \cite{MehM19} with $K=I$ and $L$ skew-symmetric, and thus only exhibiting Lagrange algebraic constraints. The dual special case are pHDAEs as defined in \cite{Sch13} with $P=I$ and $S$ symmetric, and thus only admitting Dirac algebraic constraints.

\subsection{Nonlinear port-Hamiltonian DAE systems without dissipation and external ports}
Definition~\ref{def:defph} extends to the nonlinear case as follows; see \cite{SchM20}.
\begin{definition}\label{def:defnlph}
A port-Hamiltonian system on a state manifold $\X$ is defined by a pair $(\D,\cL)$, where $\D$ is a Dirac structure on $\X$ (i.e., a subbundle of $T\X \oplus T^*\X$), and $\cL \subset T^*\X$ is a Lagrangian submanifold. Its dynamics is given by the relation
\begin{equation*}
\{(x,\dot{x}) \in T\X \mid\ \exists\, e \ \in T_x^* \X \mbox{ s.t. } (x,e) \in \cL, (-\dot{x},e) \in \D(x) \}.
\end{equation*}
\end{definition}
In the case that $\cL$ is generated by a Hamiltonian $\mathcal H: \X \to \mR$, i.e., $\mathcal L=\{(x,e) \in T^*\X \mid e= \frac{\partial \mathcal H}{\partial x}(x) \}$, we again obtain energy conservation
\begin{equation*}
\frac{d}{dt}\mathcal H = \dot x^\top \frac{\partial \mathcal H}{\partial x}(x) = 0,
\end{equation*}
since $(-\dot{x},\frac{\partial \mathcal H}{\partial x}(x)) \in \D(x)$.

Analogously to the linear case, also for singular Lagrangian submanifolds energy functions can be defined. This makes use of the representation of Lagrangian submanifolds using Morse families, cf. Appendix B, Proposition \ref{prop:nonllagrangeaug}. Indeed, under regularity conditions,  for any Lagrangian submanifold $\cL \subset T^*\X$, there locally exists a Morse family $F(x,\mu)$ such that locally
\begin{equation}
\label{eq:morse}
\cL = \{(x,e) \mid e=\frac{\partial F}{\partial x}(x,\mu), 0=\frac{\partial F}{\partial \mu}(x,\mu) \}.
\eq
Hence for the port-Hamiltonian DAE system represented by $(\D,\cL)$
it holds that
\[
\frac{d}{dt} F(x,\mu)= \dot{x}^\top \frac{\partial F}{\partial x}(x,\mu) + \dot{\mu}^\top \frac{\partial F}{\partial \mu}(x,\mu) = \dot{x}^\top e =0.
\]
Here the second equality results from the equation $0=\frac{\partial F}{\partial \mu}(x,\mu)$ in \eqref{eq:morse}, while the last equality follows from the fact that $(-\dot{x},e) \in \D$.

As in the linear case, singularity of the Dirac structure and/or of the Lagrangian submanifold defining the port-Hamiltonian system translates into the existence of (nonlinear) \emph{algebraic constraints}. In fact, if $\cL$ is singular then the inclusion $(x,e) \in \cL$ entails an algebraic constraint on $x$, and if $\D$ is singular then $(-\dot{x},e) \in \D$ defines a constraint on $e$, which via $(x,e) \in \cL$ translates into an algebraic constraint on $x$.

Analogously to \eqref{dirac}, any Dirac structure $\D$ on the $n$-dimensional state space manifold $\X$ can be locally written as
\begin{equation*}
\D(x) = \{(f,e) \in T_x\X \times T_x^*\X \mid K(x)f + L(x)e=0 \}
\end{equation*}
for certain square $n \times n$ matrix functions $K(x), L(x)$ satisfying
\begin{equation*}
K(x)L^\top (x)= - L(x) K^\top (x), \quad \rank \mat{cc} K(x) & L(x)\rix=n, \quad x \in \X.
\end{equation*}
This means that if the Lagrange submanifold $\cL$ is regular, i.e., of the form $\mathcal L=\{(x,e) \in T^*\X \mid e= \frac{\partial \mathcal H}{\partial x}(x) \}$ for a certain $\mathcal H:\X \to \mR$, then the nonlinear port-Hamiltonian DAE locally takes the form
\begin{equation*}
K(x) \dot{x} = L(x) \frac{\partial \mathcal H}{\partial x}(x)
\end{equation*}
Obviously Dirac algebraic constraints correspond to the singularity of $K(x)$.

Local coordinate expressions for a singular Lagrangian submanifold, and the corresponding Lagrange algebraic constraints, are more involved; see also \cite{SchM20}. This will be detailed, among other things, in Section \ref{sec:augmented}.

\section{Lagrange and Dirac algebraic constraints, and augmentation by Lagrange multipliers}\label{sec:augmented}

As discussed above, the algebraic constraints of general port-Hamiltonian DAE systems can be distinguished into Lagrange and Dirac algebraic constraints.
In this section we recall from \cite{SchM18} (linear case) and \cite{SchM20} (nonlinear case) how Lagrange algebraic constraints can be converted into Dirac algebraic constraints, and conversely Dirac algebraic constraints into Lagrange algebraic constraints. We show how this can be done by the use of \emph{Lagrange multipliers}; see Appendix A for the required background and statements. The multipliers will be used to augment systems with regular Lagrange and Dirac structure
without changing the energy function (Hamiltonian). This flexibility of formulating the algebraic constraints either as Lagrange or as Dirac algebraic constraints will be important in applications. This will be demonstrated by the examples in the present and next section.

In this section we restrict ourselves to systems without energy dissipation and external ports. The extension
to systems with energy dissipation and external ports will be given in Section~\ref{sec:nldiss}.
Again, we first discuss the linear case and then extend the results to nonlinear systems.

\subsection{Algebraic constraints in the linear case}\label{sec:auglin}
We start with a systematic discussion of the linear case. Throughout this subsection we consider a linear pHDAE system $(\D,\cL)$ given by \eqref{phdef}, with coordinate representation \eqref{kpls}.
%
%

\subsubsection{Singular Lagrange structure} Consider a \emph{regular} Dirac structure $\mathcal D= \graph(-J) \subset \X \times \X^*$, where $J$ is a skew-symmetric matrix
and a \emph{singular} Lagrange structure $\cL \subset \X \times \X^*$.

As detailed in Appendix A such a singular Lagrange structure can be always, see e.g. \cite{Hir79}, represented as $\mathcal L=\{(x,e)  |\ e=Qx + C_\mu \mu, C_\mu ^\top x=0\}$ for certain matrices $C_\mu$ and $Q=Q^\top$. The function $F(x,\mu)= \frac{1}{2} x^\top Q x + \mu^\top C^\top_\mu x$ is called a \emph{Morse family} for the Lagrange structure.

To obtain an augmented pHDAE representation, we consider the  \emph{augmented} state $(x,\mu) \in \X_a$, and augmented efforts $(e,e_\mu) \in \X_a^*$, together with the \emph{augmented Lagrange structure}
\[
\mathcal L_a=\graph \mat {cc} Q & C_\mu \\ C_\mu^\top & 0 \rix.
\]
Furthermore, we consider the \emph{augmented} Dirac structure
\[
\mathcal D_a=\left \{ (f,f_\mu,e,e_\mu) \in \X_a \times \X_a^* \mid \mat {cc} I & 0 \\ 0 & 0 \rix  \mat{c} f\\ f_\mu \rix=
- \mat {cc} J & 0 \\ 0 & I \rix  \mat{c} e\\ e_\mu \rix \right \}.
\]
Note that the augmented Lagrange structure $\mathcal L_a$ is \emph{regular}, while the augmented Dirac structure $\mathcal D_a$ is now \emph{singular}.
This results in an augmented pHDAE of the form \eqref{kpls} given by
\begin{equation} \label{extphdae1}
\mat {cc} I & 0 \\ 0 & 0 \rix \mat {cc} I & 0 \\ 0 & I \rix \mat{c} \dot x\\ \dot \mu \rix=
\mat {cc} J & 0 \\ 0 & I \rix
\mat{cc} Q & C_\mu\\ C_\mu^\top & 0 \rix
\mat{c} x\\ \mu \rix,
\end{equation}
with Dirac algebraic constraints, but \emph{without} Lagrange algebraic constraints. It is immediately verified that this augmented pHDAE system covers the original pHDAE in the sense that the $x$-part of the solutions of \eqref{extphdae1} coincides with the solutions of the original system.

Interestingly, the Hamiltonian of the augmented pHDAE system \eqref{extphdae1} is given by
\[
\mathcal{H}_1(x_a)=
\frac 12 \mat{c} x\\ \mu \rix^\top\mat{cc} Q & C_\mu\\ C_\mu^\top & 0 \rix
\mat{c} x\\ \mu \rix =
\frac 12 x^\top
Q x.
\]
Thus the Lagrange multipier $\mu$ does \emph{not} contribute to the Hamiltonian, and the Hamiltonian of \eqref{extphdae1} equals the Hamiltonian of the original pHDAE system.

\begin{example} \label{ex:mech}
{\rm  Consider a linear mechanical system with configuration coordinates $q$ subject to geometric constraints $G_\mu^\top {q} = 0$ for some matrix $G_\mu$. The Hamiltonian is of the form kinetic energy plus potential energy:
\[
\mathcal{H}_1(q,p)= \frac{1}{2} p^\top M^{-1}p + \frac{1}{2}q^\top \Phi q,
\]
where $p=M\dot{q}$ are the momenta, and $\Phi=\Phi^\top$ is a matrix defining the potential energy. Note that the geometric constraints $G_\mu^\top {q} = 0$ are \emph{Lagrange} algebraic constraints, corresponding to the singular Lagrange structure
\[
\cL= \{(q,p,e_q,e_p) \mid \bma e_q \\ e_p \ema - \bma \Phi & 0 \\ 0 & M^{-1} \ema
\bma q \\ p \ema \in \Img \bma G_\mu \\ 0 \ema, G_\mu^\top {q} = 0 \},
\]
which can be equivalently written as
\[
\cL= \{(q,p,e_q,e_p) \mid \exists \ \mu \mbox{ s.t. } \bma e_q \\ e_p \ema = \bma \Phi & 0 \\ 0 & M^{-1} \ema
\bma q \\ p \ema + \bma G_\mu \\ 0 \ema \mu, G_\mu^\top {q} = 0 \}.
\]
Thus, the regular augmented Lagrange structure $\cL_a$ in the augmented state  $x_a=\mat{ccc}q^\top & p^\top & \mu^\top\rix^\top$ is given as
\[
\cL_a= \left \{(q,p,\mu, e_q,e_p, e_\mu) \mid \bma e_q \\ e_p \\ e_\mu \ema = \bma \Phi & 0 & G_\mu \\ 0 & M^{-1} & 0 \\ G_\mu^\top & 0 & 0 \ema \bma q \\ p \\ \mu \ema \right \}.
\]
Correspondingly, the singular augmented Dirac structure is
\begin{equation*}
\D_a = \left \{(f_q,f_p,f_\mu, e_q,e_p,e_\mu)) \mid \bma I & 0 & 0 \\ 0 & I & 0 \\ 0 & 0 & 0 \ema
\bma f_q \\ f_p \\ f_\mu \ema + \bma 0 & I & 0 \\-I & 0 & 0 \\ 0 & 0 &I \ema \bma e_q \\ e_p \\ e_\mu \ema =0 \right \}.
\end{equation*}
Taken together, this results in the augmented pHDAE system
\begin{eqnarray*}
&&
\mat{ccc} I & 0 & 0 \\ 0 & I & 0 \\ 0 & 0 & 0 \rix
\mat{cccc} I & 0 & 0 \\ 0 & I & 0 \\ 0 & 0 & I \rix\mat{c} \dot {q} \\ \dot {p} \\ \dot \mu \rix \\  && \qquad
=
\mat{ccc} 0 & I & 0 \\ -I & 0 & 0\\ 0 & 0  & I \rix \mat{ccc} \Phi & 0 & G_\mu \\ 0 & M^{-1} & 0 \\ G_\mu^\top  & 0 & 0 \rix
\mat{c}q \\ p \\ \mu \rix.
\end{eqnarray*}
As a consequence, the Lagrange algebraic constraints $G_\mu^\top q=0$ for the original system have been converted into Dirac algebraic constraints for the augmented system.
}
\end{example}

\subsubsection{Singular Dirac structure}
Consider now instead a \emph{regular} Lagrange structure $\mathcal L= \graph(Q)$, where $Q$ is a symmetric
matrix, and a \emph{singular} Dirac structure $\D$. As detailed in Appendix A such a singular Dirac structure can be always represented as
\[
\mathcal D=\{(f,e)  |\ \exists\ e_\lambda,\ \mbox{\rm s.t. } -f=Jx+ C_\lambda e_\lambda,\  C_\lambda^\top e_\lambda=0\}
\]
for a certain skew-symmetric matrix $J$ and amatrix $C_\lambda$.
To obtain an extended pHDAE representation we consider an  augmented state vector tuple $x_a=(x,\lambda) \in \X_a$ and an augmented tuple $(f,f_\lambda,e,e_\lambda) \in \X_a \times \X_a^*$ of flows and efforts. This leads to a \emph{regular} augmented Dirac structure
\[
\mathcal D_a=\left \{(f, f_\lambda,e, e_\lambda) \mid -\bma f \\ f_\lambda \ema=\bma J & C_\lambda \\C_\lambda^\top & 0 \ema \bma e \\ e_\lambda \ema \right \},
\]
and a singular augmented Lagrange structure
\[
\mathcal L_a=\left \{ (x,\lambda,e,e_\lambda) \in \X_a \times \X_a^* | \mat {cc} Q & 0 \\ 0 & I \rix  \mat{c} x\\ \lambda \rix+
\mat {cc} I & 0 \\ 0 & 0 \rix  \mat{c} e\\ e_\lambda \rix =0 \right \}.
\]
Taken together, this gives an augmented pHDAE of the form \eqref{kpls} in the augmented state $x_a=\mat{cc} x^\top & \lambda^\top \rix^\top$ with
\begin{equation}\label{extphdae2}
\mat {cc} I & 0 \\ 0 & I\rix \mat {cc} I & 0 \\ 0 & 0 \rix \mat{c} \dot x\\ \dot \lambda \rix=
\mat {cc} J & C_\lambda \\ -C_\lambda^\top & 0\rix
\mat{cc} Q & 0 \\ 0 & I \rix
\mat{c} x\\ \lambda \rix,
\end{equation}
where, compared to the original system, the Dirac algebraic constraint $0=C_\lambda^\top e_\lambda = C_\lambda^\top Qx$ is now interpreted as a Lagrange algebraic constraint with respect to the extended state $x_a=(x,\lambda)$. Finally, note that the energy function of the augmented system \eqref{extphdae2} is given as
\[ \mathcal{H}_2(x_a) =\frac 12\mat{c} x\\ \lambda \rix^\top
\bma I & 0 \\ 0& 0 \ema \mat{cc} Q &0 \\ 0 & I \rix
\mat{c} x\\ \lambda \rix=\frac 12 x^\top
Q x.
\]
Thus, again, the Lagrange multiplier  $\lambda$ does not contribute to the Hamiltonian.
Note that this system has the form of a pHDAE system as in \cite{MehM19}.

\begin{example} \label{ex:mech2}{\rm  Consider the mechanical system  from Example~\ref{ex:mech}, but now subject to \emph{kinematic} constraints $G_\lambda^\top \dot{q} = 0$ for some matrix $G_\lambda$. Inserting $\dot q=M^{-1}p$, the kinematic constraints can be expressed as the Dirac algebraic constraint $G_\lambda^\top M^{-1}p =0$.

By d'Alembert's principle the kinematic constraints give rise to constraint forces $G_\lambda \lambda$ for some vector of Lagrange multipliers $\lambda$. This results in the extended pHDAE of the form \eqref{kpls} in  $x_a=\mat{ccc}q^\top & p^\top & \lambda^\top\rix^\top$ given by
\begin{eqnarray*}\nonumber
&&
\mat{ccc} I & 0 & 0 \\ 0 & I & 0 \\ 0 & 0 & I \rix\mat{ccc} I & 0 & 0 \\ 0 & I & 0 \\ 0 & 0 & 0 \rix \mat{c}\dot{q} \\ \dot{p} \\ \dot{\lambda} \rix \\
&&  \qquad=
\bma 0 & I & 0 \\ -I & 0 & G_\lambda \\ 0 & -G_\lambda^\top  & 0 \ema \mat{ccc} \Phi & 0 & 0 \\ 0 & M^{-1} & 0 \\ 0 & 0 & I \rix
\mat{c}q \\ p \\ \lambda \rix,
\end{eqnarray*}
where again the Hamiltonian is
\[
\mathcal{H}_2(q,p,\lambda)= \frac{1}{2} p^\top M^{-1}p + \frac{1}{2}q^\top \Phi q.
\]
}
\end{example}

\subsection{Algebraic constraints in the nonlinear case}
In this subsection we  extend the results of the previous subsection to the nonlinear case.
\subsubsection{Singular Dirac structure and regular Lagrangian submanifold}\label{sec:singdirac}
Consider a port-Hamiltonian DAE system $(\D,\cL)$, on a state space manifold $\X$, where the Dirac structure $\D$ is \emph{singular}, while the Lagrangian submanifold $\cL$ is \emph{regular}, i.e., locally $\cL=\{(x,e) \in T^*\X \mid e= \frac{\partial  \mathcal{H}}{\partial x}(x)\}$ for some generating function $\mathcal{H}$ (the Hamiltonian). Then by Proposition \ref{prop:nonldiracaug} in Appendix A,
the Dirac structure can be locally written as
\begin{eqnarray*}
\nonumber
\D(x) &=& \left \{(f,e) \in T_x\X \times T_x^*\X \mid \exists\, e_\lambda \in \Lambda^* \mbox{ s.t. }-f = J(x)e + C_\lambda(x) e_\lambda,\right .\\ &&\qquad \left . 0  =  C_\lambda^\top(x)e  \right \},
\end{eqnarray*}
for some skew-symmetric matrix function $J(x)$ and a matrix function $C_\lambda(x)$. This means that on the augmented state space manifold $\X_a:= \X \times \Lambda$ one defines the \emph{regular} Dirac structure $\D_a$ as the graph of the skew-symmetric matrix function
\begin{equation*}
J_a = \begin{bmatrix} -J(x) & -C_\lambda(x) \\[2mm]
C_\lambda^\top(x) & 0 \end{bmatrix}.
\end{equation*}
On the other hand, the regular Lagrangian submanifold $\cL \subset T^*\X$ is augmented to a \emph{singular} Lagrangian submanifold $\cL_a \subset T^*\X_a$ defined as
\begin{equation}
\label{L_a}
\cL_a := \{ (x,\lambda, e, e_\lambda) \in T^* \X_a \mid e= \frac{\partial \mathcal H}{\partial x}(x), \lambda = 0 \},
\end{equation}
corresponding to the Lagrange algebraic constraints $0=\lambda$. This results in the following proposition.
\begin{proposition} Consider a port-Hamiltonian DAE system $(\D,\cL)$ on a state space manifold $\X$, where the Dirac structure $\D$ is \emph{singular}, while the Lagrangian submanifold $\cL$ is \emph{regular}. Then it can be augmented to a port-Hamiltonian DAE system $(\D_a,\cL_a)$ on the augmented space manifold $\X_a$ with regular Dirac structure and singular Lagrangian submanifold. Furthermore, the $x$-dynamics of this augmented system is the same as of the original system.
\end{proposition}
Comparing with the linear case (i.e., the definition of the matrices $P,S$ representing the singular Lagrange structure), we note that the singular Lagrangian submanifold $\cL_a$ defined in \eqref{L_a} can be also represented as the solution set of
\begin{equation*}
\bma I & 0  \\[2mm] 0 & 0  \ema
\bma e  \\[2mm] e_\lambda \ema=
\bma \frac{\partial \mathcal H_a}{\partial x}(x,\lambda) \\[2mm] \frac{\partial \mathcal H_a}{\partial x}(x,\lambda) \ema,
\end{equation*}
where
\begin{equation*}
\mathcal{H}_a(x,\lambda):= \mathcal H(x) + \frac 12 \lambda^\top \lambda.
\end{equation*}
Thus, Dirac algebraic constraints can always be converted into Lagrange algebraic constraints by augmenting with Lagrange multipliers. Furthermore, note that for the augmented pHDAE system
\[
\frac{d}{dt} \mathcal H(x)= \dot{x}^\top \frac{\partial \mathcal H}{\partial x}(x)= \dot{x}^\top e=0.
\]
Indeed $\dot{x} = J(x) e + C_\lambda(x)e_\lambda$, and therefore
\[
\dot{x}^\top e = e^\top J^\top(x) e + \lambda^\top C_\lambda^\top (x) e = 0,
\]
by the skew-symmetry of $J(x)$ and since $\lambda=0$. Thus $\mathcal H$ (independent of $\lambda$) is still the Hamiltonian of the augmented system $(\D,\cL)$.

\subsubsection{Singular Lagrangian submanifolds}
Alternatively, consider a nonlinear port-Hamiltonian DAE system $(\D,\cL)$, where the Lagrangian submanifold $\cL$ is singular. Then, using Proposition \ref{prop:nonllagrangeaug},
$\cL \subset T^*\X$ can be locally represented, using a Morse family $F(x,\mu)$, as
\begin{equation}
\label{morse}
\cL= \{ \left(x, \frac{\partial F}{\partial x}(x, \mu) \right) \mid \frac{\partial F}{\partial \mu}(x, \mu)=0 \}.
\end{equation}
This leads to the regular augmented Lagrangian submanifold
\begin{equation*}
\cL_a = \{(x,\lambda,e,e_\mu) \mid e= \frac{\partial F}{\partial x}(x,\mu), e_\mu = \frac{\partial F}{\partial \lambda}(x,\mu) \}
\end{equation*}
on the augmented space of points $(x,\mu)$.

Furthermore, the Dirac structure $\D$ on $\X$ may be augmented with the trivial Dirac structure $\{(f_{\mu},e_{\mu}) \mid e_{\mu}=0 \}$, so as to obtain an augmented Dirac structure $\D_a$. This defines an {\it extended} port-Hamiltonian system $(\D_a,\cL_a)$ with an {\it explicit} energy function $F(x,\mu)$, and thus without any Lagrange algebraic constraints.

This is summarized in the following proposition.
\begin{proposition}
Consider a port-Hamiltonian DAE system $(\D,\cL)$ with Lagrange algebraic constraints. Represent $\cL$ by a Morse family $F(x,\mu)$ on $\X_a= \X \times \Lambda$ as in \eqref{morse}, and define the augmented Dirac structure $\D_a$ as the product of $\D$ with the Dirac structure $\{(f_{\mu},e_{\lambda}) \mid e_{\mu}=0 \}$. Then $(\D,\cL)$ has been transformed into the port-Hamiltonian system $(\D_a,\cL_a)$ with state space $\X_a$, in such a way that its $x$-dynamics is the same as that of the original port-Hamiltonian DAE system.
\end{proposition}

\begin{example}[Optimal control \cite{BarPM15}]\label{singular}
{\rm Consider the optimal control problem of minimizing a cost functional $\int_{t_0}^{t_f}  K(q(t),u(t)) dt$ for a control system $\dot{q}=f(q,u)$, with $q: [t_0,t_f]\to  \mR^n$, $u [t_0,t_f]\to \mR^m$, $q(t_0) \in \mR^n$. Define the \emph{optimal control Hamiltonian}
\begin{equation*}
\mathcal H(q,p,u)=  p^\top f(q,u) - K(q,u) ,
\end{equation*}
with $p \in \mR^n$ the co-state vector. Application of Pontryagin's Maximum principle leads to the consideration of the port-Hamiltonian DAE system on the space with coordinates $(q,p,u)$, given as
\begin{equation}
\label{o1}
\bma I & 0 & 0 \\[2mm] 0 & I & 0 \\[2mm] 0 & 0 & 0 \ema
\begin{bmatrix}
\dot{q} \\[2mm] \dot{p} \\[2mm] \dot{u}
\end{bmatrix} =
\begin{bmatrix}
0 & I_n & 0 \\[2mm]
-I_n & 0 & 0 \\[2mm]
0 & 0 & I_m
\end{bmatrix}
\begin{bmatrix}
\frac{\partial \mathcal H}{\partial q}(q,p,u) \\[2mm]
\frac{\partial \mathcal H}{\partial p}(q,p,u) \\[2mm]
\frac{\partial \mathcal h}{\partial u}(q,p,u)
\end{bmatrix}.
\end{equation}
The underlying Dirac structure is given as
\begin{equation*}
\D =\left \{ \left( \begin{bmatrix} f_q \\[2mm] f_p \\[2mm] f_u \end{bmatrix} , \begin{bmatrix} e_q \\[2mm] e_p \\[2mm] e_u \end{bmatrix}\right) \mid f_q = - e_p, \, f_p = e_q, \, e_u=0\right \} ,
\end{equation*}
i.e., the direct product of the Dirac structure on the $(q,p)$ space given by the graph of the canonical skew-symmetric map $\begin{bmatrix} 0 & -I \\ I & 0 \end{bmatrix}$ from $\begin{bmatrix} e_q \\ e_p \end{bmatrix}$ to $\begin{bmatrix} f_q \\ f_p \end{bmatrix}$, with the trivial Dirac structure $\{(f_u,e_u) \mid e_u=0 \}$. The resulting Dirac algebraic constraints are $\frac{\partial \mathcal H}{\partial u}(q,p,u) =0$.

System \eqref{o1} can be equivalently rewritten as a port-Hamiltonian system only involving the $(q,p)$ variables, with {\it implicit} energy storage relation given by the Lagrangian submanifold
\begin{equation}
\label{o2}
\cL = \left \{\left(
\begin{bmatrix}
q \\[2mm] p
\end{bmatrix},
\begin{bmatrix}
e_q \\[2mm] e_p
\end{bmatrix} \right)
\mid
\exists\, u \mbox{ s.t. }
\begin{bmatrix}
e_q \\[2mm] e_p
\end{bmatrix}  =
\begin{bmatrix}
\frac{\partial \mathcal{H}}{\partial q}(q,p,u) \\[2mm]
\frac{\partial \mathcal{H}}{\partial p}(q,p,u)
\end{bmatrix} , \; \frac{\partial \mathcal{H}}{\partial u}(q,p,u)=0.
\right \}
\end{equation}
Thus, the optimal control Hamiltonian $\mathcal H(q,p,u)$ defines a Morse family for $\cL$ (with $u$ taking the role of the Lagrange multipliers), and the conversion of \eqref{o2} into \eqref{o1} is an example of conversion of a Lagrange algebraic constraint into a Dirac algebraic constraint, see  \cite{SchM20} for details.
}
\end{example}

An important special case of singular Lagrangian submanifolds is the following situation. Consider a Lagrangian submanifold $\cL \subset T^*\X$ with natural coordinates $(x,e)$, with a Morse family of the separable form
\begin{equation*}
F(x,\mu)= \mathcal H(x) + \mu^\top c(x)
\end{equation*}
for some Hamiltonian $\mathcal H:\X \to \mR$ and constraint mapping $c: \X \to \mR^k$. This means that $\cL$ is given as
\begin{equation*}
\cL = \{(x,e) \mid \exists\, \mu \mbox{ s.t. } e= \frac{\partial \mathcal H}{\partial x}(x) + C_\mu(x)\mu,\ c(x)=0 \},
\end{equation*}
where $C_\mu(x):=\frac{\partial c^\top}{\partial x}(x)$ and
$\mu$ are Lagrange multipliers.
Combined with a general Dirac structure $\D$ on $\X$ with coefficient functions $K(x),L(x)$, this leads to a port-Hamiltonian DAE system
\begin{equation}
\label{explicit}
\mat{cc} K(x) & 0 \\ 0 &  0\rix \mat{c} \dot{x} \\ \dot \mu \rix =  \mat{cc} L(x) \frac{\partial \mathcal H}{\partial x}(x) + C_\mu (x) \mu  \\ c(x)\rix,
\end{equation}
or equivalently
\begin{equation*}
\mat{cc} K(x) & 0 \\ 0 &  0\rix \mat{c} \dot{x} \\ \dot \mu \rix = \mat{cc} L(x) & 0 \\ 0 & I \rix \mat{cc} \frac{\partial F}{\partial x}(x,\mu)\\ \frac{\partial F}{\partial \mu}(x,\mu)\rix.
\end{equation*}
A special case of \eqref{explicit} is
\begin{equation}
\label{explicit2}
\mat{cc} I & 0 \\ 0 &  0\rix \mat{c} \dot{x} \\ \dot \mu \rix =  \mat{cc} J(x) \frac{\partial \mathcal H}{\partial x}(x) + C_{\mu} (x) \mu  \\ c(x)\rix,
\end{equation}
where $J(x)=-J^\top (x)$,
and we obtain the following energy balance
\begin{equation}
\begin{array}{rcl}
\frac{d}{dt} \mathcal H &=&
\dot{x}^\top \frac{\partial \mathcal H}{\partial x}(x)  = - \left[\frac{\partial\mathcal  H}{\partial x}(x) + C_\mu (x) \mu \right]^\top J(x) \frac{\partial \mathcal H}{\partial x}(x)\\[2mm]
&=&
- \mu^\top C^\top_\mu (x) J(x) \frac{\partial \mathcal H}{\partial x}(x).
\end{array}
\label{balance}
\end{equation}
On the other hand $\mu$ is chosen such that the constraint $c(x)=0$ remains satisfied. Hence by differentiation, $\mu$ satisfies
\begin{eqnarray*}
0&=& C_\mu ^\top (x) \dot{x} = C_\mu^\top (x) J(x) \left[\frac{\partial \mathcal H}{\partial x}(x) + C_\mu(x) \mu \right]\\ &=& C_\mu ^\top (x) J(x) \frac{\partial \mathcal H}{\partial x}(x) + C_\mu^\top (x) J(x)C_\mu(x) \mu.
\end{eqnarray*}
Together with \eqref{balance} this yields energy conservation, i.e.,
\begin{equation*}
\frac{d}{dt} \mathcal H = - \mu^\top C_\mu^\top (x) J(x) C_\mu(x) \mu =0.
\end{equation*}
A special case
of \eqref{explicit2}
are mechanical systems with position constraints $g(q)=0$, where $x=(q,p)$. This will be discussed in the next section.

\begin{remark}
{\rm
The developments in this subsection also form the starting point for the theory of \emph{Dirac brackets}, as developed in the 1950s by Paul Dirac, see e.g. \cite{Dir58}. In this context the constraints $c(x)=0$ arise from non-surjectivity of the map $(q,\dot{q}) \mapsto (q, p= \frac{\partial l}{\partial \dot{q}}(q,\dot{q}))$ from $TQ$ to $T^*Q$, for a singular Lagrangian function $L(q,\dot{q})$. This means that the Hamiltonian is $\mathcal H(x) + \mu^\top c(x)$, where $\mu$ is undetermined. By possibly differentiating the \emph{primary} constraints $c(x)=0$ to \emph{secondary} constraints one is led to a constraint submanifold of the state space $\X = T^*Q$. Under a regularity assumption (the decoupling matrix has full rank), the system \eqref{explicit2} restricted to this constraint manifold defines a unique Hamiltonian vector field  with respect to a Poisson bracket that is called the Dirac bracket. See \cite{NvdS90} (Chapter 12) for details.
}
\end{remark}

\section{Nonlinear mechanical systems}\label{sec:nlmech}
In this section we will look specifically at algebraic constraints that occur in nonlinear mechanical systems. This includes primarily the study of geometric and kinematic constraints, but also the classical (although) relatively unknown) Livens principle.

\subsection{Hamilton's equations and Livens principle}
A standard way to derive the equations of motion of a mechanical system is to start from a Lagrangian function $L(q,\dot{q})$, where $q$ are the configuration coordinates of the mechanical system and $\dot{q}$ the corresponding generalized velocities. Using variational calculus for optimizing the \emph{action} integral $\int L(q(t),\dot{q}(t)) dt$ one derives the \emph{Euler-Lagrange equations}
\begin{equation}
\label{EL}
\frac{d}{dt} \frac{\partial L} {\partial \dot{q}}(q,\dot{q}) = \frac{\partial L}{\partial q}(q,\dot{q})
\end{equation}
describing the second-order equations of motion.
Next one defines the Hamiltonian function $\mathcal H(q,p)$ as the \emph{partial Legendre transform}, see e.g. \cite{Arn78}, (assuming it exists) of $L(q,\dot{q})$ with respect to $\dot{q}$, with $p=\frac{\partial L}{\partial \dot{q}}$ the vector of \emph{momenta}. As a result, the Euler-Lagrange equations \eqref{EL} transform into the classical \emph{Hamiltonian equations}
\begin{equation}
 \begin{bmatrix}  \dot q \\ \dot p \end{bmatrix} = \begin{bmatrix}  0 & I \\ -I & 0 \end{bmatrix}\begin{bmatrix}  \frac{\partial \mathcal H}{\partial q}(q,p) \\  \frac{\partial \mathcal H}{\partial p}(q,p) \end{bmatrix}. \label{mechclassical}
\end{equation}
For example, if $L(q,\dot{q}) = \frac12  \dot{q}^\top M(q)
\dot{q}- V(q)$ (kinetic minus potential energy), for some invertible mass matrix $M(q)$, then the Hamiltonian is given as as the total energy
$\mathcal H(q,p)=\frac12  p^\top M^{-1}(q) p +V(q)$.

An alternative view is provided by the classical \emph{Livens principle} \cite{Liv20}, see e.g. \cite{KinBS23}. This avoids the explicit Legendre transform by introducing the velocity $v$ as an \emph{extra} variable, and considering the \emph{augmented Lagrangian}
\[
L_1(q,v) :=L(q,v) + p^\top (\dot{q} -v).
\]
Thus $\dot{q}$ in the Lagrangian $L(q,\dot{q})$ is replaced by $v$, while interpreting the equality $v=\dot{q}$ as a \emph{constraint} $\dot{q}-v=0$ with Lagrange multiplier $p$. Application of variational calculus to this augmented Lagrangian $L_1$ (and exploiting affine dependence on $\dot{q}$) yields the following system
\begin{equation}
\begin{bmatrix} I & 0 & 0 \\ 0 & I & 0 \\ 0 & 0 & 0 \end{bmatrix}
 \begin{bmatrix}  \dot q \\ \dot p \\ \dot v \end{bmatrix} = \begin{bmatrix} 0 & I & 0 \\ -I & 0 & 0 \\ 0 & 0 & I\end{bmatrix}\begin{bmatrix}  \frac{\partial \mathcal H_1}{\partial q} \\  \frac{\partial \mathcal H_1}{\partial p} \\ \frac{\partial \mathcal H_1}{\partial v} \end{bmatrix}, \label{mechlivens}
\end{equation}
in the extended set of state variables $q,p,v$, with Hamiltonian
\begin{equation*}
\mathcal H_1 (q,p,v):= p^\top v - L(q,v).
\end{equation*}
This constitutes a port-Hamiltonian DAE system with Dirac algebraic constraints
$0=\frac{\partial \mathcal H_1}{\partial v}(q,v)=p - \frac{\partial L}{\partial v}$, which are replacing the \emph{definition} $p=\frac{\partial L}{\partial \dot{q}}$ in the standard approach.
If we can solve these algebraic constraints for $v=\tilde v(q,p)$, then we are back to the standard case \eqref{mechclassical}, with Hamiltonian $\mathcal H(q,p)=\mathcal H_1(q,\tilde v(q,p),p)$. However, solvability of the algebraic constraints (or, equivalently, existence of the partial Legendre transform of $L$) is \emph{not} required for deriving \eqref{mechlivens}.
\begin{remark}
{\rm The
Livens principle can be regarded as a special case of the Pontryagin maximum principle \cite{BolGMP62} in optimal control (as was already indicated in \cite{KinBS23}). In fact, consider the optimal control problem of minimizing a cost criterion $\int K(x(t),u(t) dt$ with respect to dynamics $\dot{x}=f(x,u)$; see Example \ref{singular}. Then the first-order necessary conditions of Pontryagin's maximum principle amount to the port-Hamiltonian DAE system
\begin{equation}
\label{optimal}
\begin{bmatrix} I & 0 & 0 \\ 0 & I & 0 \\ 0 & 0 & 0 \end{bmatrix}
 \bma \dot{x} \\ \dot{p} \\ \dot{u} \ema
 = \begin{bmatrix} 0 & I & 0 \\ -I & 0 & 0 \\ 0 & 0 & I\end{bmatrix}
 \begin{bmatrix}  \frac{\partial \mathcal H_{opt}}{\partial q} \\  \frac{\partial \mathcal H_{opt}}{\partial p} \\ \frac{\partial \mathcal H_{opt}}{\partial u} \end{bmatrix},
 \end{equation}
with $\mathcal H_{opt}(x,u,p) := p^\top f(x,u) - L(x,u)$ being the \emph{optimal control Hamiltonian}. Thus, seen from this angle the Livens principle corresponds to the special case $f(x,u)=u$. All this is closely related to the observation made in \cite{suss97}, where \eqref{optimal} is referred to as 'Hamilton's equations as he should have written them'. Said otherwise, from our point of view 'Hamilton's equations are actually port-Hamiltonian DAEs'.
}
\end{remark}

\subsection{Mechanical systems with geometric constraints}
Next consider a nonlinear mechanical system with Hamiltonian $\mathcal H$ as in \eqref{mechclassical}, and with \emph{geometric constraints} $g(q)=0$ for some mapping $g: \mathcal Q \to \mR^k$, where $\mathcal Q$ is the state space or manifold of the $q$ vectors. The equations of motion are given as
\begin{equation}
\label{mechgeoml}
\begin{array}{rcl}
\dot q &=& \frac{\partial \mathcal H}{\partial p}(q,p), \\
\dot p &=&-\frac{\partial  \mathcal H}{\partial q}(q,p) - G_\mu(q)  \mu,  \\
0 & = &  g(q),
\end{array}
\end{equation}
with the matrix function $G_\mu(q)=\frac{ \partial g^\top}{\partial q}(q)$ defining the \emph{constraint forces} $G_\mu(q)  \mu$, where  $\mu$ are appropriate Lagrange multipiers. This constitutes a port-Hamiltonian DAE system with state variables $q,p$ and \emph{Lagrange} algebraic constraints $g(q)=0$. Indeed, the submanifold
\begin{equation}
\mathcal{L}:= \{(q,p,e_q,e_p) \mid g(q)=0, e_q-\frac{\partial \mathcal H}{\partial q}(q,p) \in \Img G(q), e_p= \frac{\partial \mathcal H}{\partial p}(q,p) \}
\end{equation}
is a Lagrangian submanifold of $T^*\X$, where $\X=T^*\mathcal Q$ is the state space with natural coordinates $q,p$. Note that the Dirac structure $\D$ is simply given as the graph of $-J= \bma 0 & -I \\ I & 0 \ema$.

The Lagrange algebraic constraints are converted into Dirac algebraic constraints by introducing the Lagrange multipliers $\mu$ and considering the Morse family
\[
\mathcal H_2(p,q,\mu) :=\mathcal H(q,p)+\mu^\top g(q).
\]
Then \eqref{mechgeoml} is rewritten as the augmented port-Hamiltonian DAE system
\begin{equation}
\begin{bmatrix} I & 0 & 0 \\ 0 & I & 0 \\ 0 & 0 & 0\end{bmatrix}
 \begin{bmatrix}  \dot q \\ \dot p \\ \dot \mu \end{bmatrix} = \begin{bmatrix} 0 & I & 0 \\ -I & 0 & 0 \\ 0 & 0 & I\end{bmatrix}\begin{bmatrix}  \frac{\partial \mathcal H_2}{\partial q} \\  \frac{\partial \mathcal H_2}{\partial p} \\ \frac{\partial \mathcal H_2}{\partial \mu} \end{bmatrix},   \label{mechmorse}
\end{equation}
with \emph{Dirac} algebraic constraints $\frac{\partial \mathcal H_2}{\partial \mu} =- g(q)=0$. These equations can be more explicitly rewritten as
\begin{equation*}
\begin{bmatrix} I & 0 & 0 \\ 0 & I & 0 \\ 0 & 0 & 0\end{bmatrix}
 \begin{bmatrix}  \dot q \\ \dot p \\ \dot \mu \end{bmatrix} = \begin{bmatrix} 0 & I & 0 \\ -I & 0 & 0 \\ 0 & 0 & I\end{bmatrix}\begin{bmatrix}  \frac{\partial \mathcal H}{\partial q}(q,p) \\  \frac{\partial \mathcal H}{\partial p}(q,p) \\ g(q) \end{bmatrix}  - \bma 0 \\ G_\mu(q) \\ 0 \ema \mu 
\end{equation*}
thus indeed recovering \eqref{mechgeoml}.
This port-Hamiltonian DAE formulation of mechanical systems with geometric constraints can be \emph{combined} with Livens principle by considering the further augmented set of state variables $q,p,v,\mu$, and the Hamiltonian
\[
\mathcal H_4(q,p,v,\mu) :=\mathcal H_1(q,p,v) + \mu^\top g(q)= p^\top v - L(q,v) + \mu^\top g(q).
\]
The resulting port-Hamiltonian formulation takes the form
\begin{equation*}
\begin{bmatrix} I & 0 & 0 &0\\ 0& I& 0& 0\\ 0 &0 & 0 & 0 \\ 0 & 0 & 0&0\end{bmatrix}
 \begin{bmatrix}  \dot q \\  \dot p\\ \dot v \\ \dot \mu \end{bmatrix} = \begin{bmatrix}0&  I & 0 & 0 \\ -I & 0 & 0 & 0 \\ 0& 0 & I & 0\\ 0 & 0 & 0 &I\end{bmatrix}\begin{bmatrix}  \frac{\partial \mathcal H_4}{\partial q} \\  \frac{\partial \mathcal H_4}{\partial p}\\ \frac{\partial \mathcal H_4}{\partial  v} \\ \frac{\partial \mathcal H_4}{\partial \mu} \end{bmatrix}. 
\end{equation*}
%
This leads to two sets of Dirac algebraic constraints, namely $\frac{\partial \mathcal H_1}{\partial v}(q,v)= p - \frac{\partial L}{\partial v}=0$ and $g(q)=0$.
Solving (if possible!) the Livens algebraic constraint $\frac{\partial \mathcal H_1}{\partial v}(q,v)= p - \frac{\partial L}{\partial v}=0$ for $v$ leads back to \eqref{mechmorse}.

\subsection{Kinematic constraints }
Kinematic constraints for mechanical systems are of the form $A^\top (q)\dot{q}=0$ for some $n \times k$ matrix function $A(q)$. In this case, applying d'Alembert's principle, \eqref{mechclassical} extends to the port-Hamiltonian DAE system
\begin{equation}
\begin{array}{rcl}
 \begin{bmatrix}  \dot q \\ \dot p \end{bmatrix} &=& \begin{bmatrix}  0 & I \\ -I & 0 \end{bmatrix}\begin{bmatrix}  \frac{\partial \mathcal H}{\partial q}(q,p) \\  \frac{\partial \mathcal H}{\partial p}(q,p) \end{bmatrix} - \bma 0 \\ A(q) \ema e_\lambda,  \\[5mm]
 0 & = & A^\top (q) \frac{\partial H}{\partial p}(q,p),
 \end{array}
 \label{mechkinematic}
\end{equation}
with $A(q) e_\lambda$ the vector of \emph{constraint forces}, and $\mathcal H(q,p)$ the energy as in \eqref{mechclassical}. Here, the last line represents Dirac algebraic constraints resulting from the singular Dirac structure
\begin{equation*}
\begin{array}{l}
\mathcal D= \left \{(f_q,f_p,e_q,e_p) \mid \exists\, e_\lambda \mbox{ s.t. } \bma f_q \\ f_p \ema = \bma 0 & -I \\ I & 0 \ema \bma e_q \\ e_p \ema - \bma 0 \\ A(q) \ema e_\lambda, \right .\\ \qquad \qquad \left. A^\top (q) e_p=0 \right \}.
\end{array}
\end{equation*}
As before, these Dirac algebraic constraints can be converted into Lagrange algebraic constraints by extending the state to $(q,p, \lambda)$, with $\dim \lambda = \dim e_\lambda =k$. The augmented Dirac structure is given as the graph of the skew-symmetric map
\[
J_a(q) := \bma 0 & -I & 0 \\ I & 0 & A(q) \\ 0 & -A^\top(q) & 0 \ema.
\]
Furthermore, the singular augmented Lagrangian submanifold $\cL_a$ is given as
\[
\cL_a= \{(q,p, \lambda , e_q,e_p, e_{\lambda}) \mid \lambda=0,e_q=\frac{\partial \mathcal H}{\partial q}(q,p), e_p= \frac{\partial \mathcal H}{\partial p}(q,p) \}.
\]
The resulting augmented port-Hamiltonian DAE system takes the form
\begin{equation*}
\bma I & 0 & 0 \\ 0 & I & 0 \\ 0 & 0 & 0 \ema
\bma \dot{q} \\ \dot{p} \\ \dot{\lambda} \ema =
\bma 0 & I & 0 \\ -I & 0 & -A(q) \\ 0 & A^\top(q) & 0 \ema
\bma \frac{\partial \mathcal H_a}{\partial q}(q,p, \lambda) \\ \frac{\partial \mathcal H_a}{\partial p}(q,p,\lambda) \\ \frac{\partial \mathcal H_a}{\partial \lambda}(q,p, \lambda) \ema,
\end{equation*}
where $\mathcal H_a(q,p,\lambda) = \mathcal H(q,p) + \frac 12 \lambda^\top \lambda$.

Finally, the treatment of kinematic constraints can be also combined with the Livens principle. This is done by adding variables $v$ and replacing the Hamiltonian $\mathcal H (q,p)$ as in \eqref{mechclassical} by the Hamiltonian $\mathcal H_1(q,p,v)= p^\top v - L(q,v)$ introduced before. This results in the following port-Hamiltonian DAE system with state variables $(q,p,v)$ given as
\begin{equation*}
\begin{bmatrix} I & 0 & 0 \\ 0 & I & 0 \\ 0 & 0 & 0 \end{bmatrix}
 \begin{bmatrix}  \dot q \\ \dot p \\ \dot v \end{bmatrix} = \begin{bmatrix} 0 & I & 0 \\ -I & 0 & 0 \\ 0 & 0 & I\end{bmatrix}\begin{bmatrix}  \frac{\partial \mathcal H_1}{\partial q} \\  \frac{\partial \mathcal H_1}{\partial p} \\ \frac{\partial \mathcal H_1}{\partial v} \end{bmatrix} - \bma 0 \\ A(q) \\ 0 \ema e_\lambda, \;0= A^\top(q) \frac{\partial \mathcal H_1}{\partial p}(q,p).
\end{equation*}
This system has, on the one hand, Dirac algebraic constraints of the form $0= A^\top (q) \frac{\partial \mathcal H_1}{\partial p}(q,p)$, and, on the other hand, Lagrange algebraic constraints $0=\frac{\partial \mathcal H_1}{\partial v}(q,p)$.

\subsection{Differentiating geometric to kinematic constraints}
Another way that kinematic constraints arise is by  differentiating  geometric constraints $g(q)=0$  with respect to time, resulting in kinematic constraints $G_\lambda^\top(q) \dot{q}=0$, where $G_\lambda(q):= \frac{\partial g^\top}{\partial q}(q)$. These kinematic constraints are of a special type; classically called \emph{holonomic constraints}. (Conversely, kinematic constraints $A^\top(q)\dot{q}=0$ for which there does \emph{not} exist an invertible matrix $U(q)$ and map $g(q)$ such that $A(q)U(q)= \frac{\partial g^\top}{\partial q}(q)$ are called \emph{non-holonomic}.)

Of course, by differentiating geometric constraints $g(q)=0$ to kinematic constraints $G_\lambda^\top(q) \dot{q}=0$ one loses information. An option is to \emph{retain} the geometric constraints, next to the kinematic constraints given by the differentiated geometric constraints. This combines \eqref{mechgeoml} and \eqref{mechkinematic}, and thus leads to the port-Hamiltonian DAE system
\begin{equation*}
\begin{array}{rcl}
 \begin{bmatrix}  \dot q \\ \dot p  \end{bmatrix} &=& \begin{bmatrix} 0 & I  \\ -I & 0  \end{bmatrix}
 \begin{bmatrix}  \frac{\partial \mathcal H}{\partial q}(q,p) \\  \frac{\partial \mathcal H}{\partial p}(q,p) \end{bmatrix}  + \bma 0 \\ G_\mu(q)  \ema \mu + \bma 0 \\ G_\lambda(q)  \ema e_\lambda, \\[2mm]
 0 & = & g(q), \\[2mm]
 0 & = &G^\top_\lambda (q) \frac{\partial \mathcal H}{\partial p}(q,p).
 \end{array}
\end{equation*}
However, since $G_\lambda(q)= G_\mu(q)=:G(q)$, the two constraint forces can be taken together as the \emph{single} constraint force $G(q)(\mu + e_\lambda)$.

Alternatively, a classical approach in the numerical solution of mechanical systems with geometric constraints is to interpret the differentiated geometric constraints
$G_\lambda^\top (q)\frac{\partial \mathcal H}{\partial  p}(q,p)=:g^v(q,p)$ as new 'geometric' constraints $g^v(q,p)=0$ in the $q,p$ variables. From a port-Hamiltonian perspective this means that we do not regard the differentiated geometric constraints as kinematic constraints $G_\lambda^\top (q) \dot q=0$ (and thus as Dirac algebraic constraints), but instead as \emph{Lagrange} algebraic constraints, directly constraining the state variables $q,p$ via another Lagrange multiplier $\nu$. This is the common approach in the Gear-Gupta-Leimkuhler index reduction \cite{GeaLG85} and has been shown to be a variation of the Livens principle in \cite{KinBS23}.
It is associated with an extended Morse family $\mathcal H_5(q,v,p,\mu,\nu):=\mathcal H_1(q,v,p)+\mu^\top g(q)  + \nu^\top g^v(q,p)$, and leads to the port-Hamiltonian DAE system
\begin{equation*}
\begin{bmatrix} I & 0 & 0 &0& 0\\ 0& I& 0& 0& 0\\ 0 &0 & 0 & 0 & 0 \\ 0 & 0 &0 & 0&0\\ 0 & 0 &0 & 0&0\end{bmatrix}
 \begin{bmatrix}  \dot q \\ \dot p \\ \dot v\\ \dot \mu \\\dot \nu \end{bmatrix} = \begin{bmatrix}0 & I & 0 & 0 & 0 \\ -I&  0 & 0 & 0 & 0 \\ 0& 0 & I & 0 & 0 \\ 0 & 0 & 0 &I & 0\\ 0 & 0 & 0 &0 &I \end{bmatrix}\begin{bmatrix}  \frac{\partial \mathcal H_5}{\partial q} \\  \frac{\partial \mathcal H_5}{\partial p}\\ \frac{\partial \mathcal H_5}{\partial  v} \\ \frac{\partial \mathcal H_5}{\partial \mu} \\ \frac{\partial \mathcal H_5}{\partial \nu }\end{bmatrix},
\end{equation*}
where $ \frac{\partial \mathcal H_5}{\partial q} = -\frac{\partial L}{\partial q} +G_(q)\mu +
\frac{\partial  g^v(q,p)}{\partial q}^\top \nu $, $ \frac{\partial \mathcal H_5}{\partial  p} = v+ \frac{\partial }{\partial p}(g^v(q,p)^\top \nu)$,
$\frac{\partial \mathcal H_5}{\partial v}=p-\frac{\partial L}{\partial v}$,
$ \frac{\partial \mathcal H_5}{\partial  \mu}=g(q)$, and
$ \frac{\partial \mathcal H_5}{\partial \nu}=g^v(q,p)$.

Substitution of $\dot q=v+\frac{\partial }{\partial p} g^v(q,p)^\top \nu $ yields
\begin{equation}\label{starstar}
G^\top(q)(v+\frac{\partial }{\partial q} g^v(q,p)^\top \nu)=0.
\end{equation}
By using $g^v(q,p)=G^\top(q)\frac{\partial \mathcal H}{\partial  p}(q,p) =0$, then \eqref{starstar}
 implies that
 \[
 G_\mu^\top(q)\frac{\partial }{\partial p} (G_\mu^\top(q)\frac{\partial \mathcal H(q,p)}{\partial p})^\top \nu)=0,
 \]
 or equivalently
\[
 G_\mu^\top(q)\frac{\partial^2 \mathcal H(q,p)}{\partial p^2}  G_\mu(q)\nu=0.
 \]
 If $G_\mu(q)$ has full rank and $\frac{\partial^2 \mathcal H(q,p)}{\partial p^2}$ is invertible then this implies that actually $ \nu=0$; see \cite{KinBS23}.

Thus, from a mathematical modeling point of view, the inclusion of the differentiated geometric constraints as additional Lagrange constraints $0=g^v(q,p)$ does not change the dynamics. However, in numerical integration methods these additional constraints \emph{can} be used to guarantee that the (hidden) kinematic constraints are satisfied, by checking the deviation of the associated Lagrange multiplier $\nu$ from being $0$.

\section{Port-Hamiltonian DAE systems with dissipation and ports.}\label{sec:nldiss}
In the previous sections we have considered port-Hamiltonian DAE systems that are \emph {energy-preserving}, and \emph{without external ports} (i.e., no inputs and outputs). External ports are immediately included by extending the Dirac structure; see later. In order to deal with energy-dissipation, the standard way in port-Hamiltonian theory is to first extend the Dirac structure to include a port linking to energy-dissipating elements, and then to \emph{close} this port by an \emph{energy-dissipation relation}. Together this leads, see \cite{SchM20}, to the following definition of dissipative port-Hamiltonian DAE systems.

\begin{definition} A nonlinear dissipative port-Hamiltonian DAE system is defined by a triple $(\D,\cL,\cR)$, where
\begin{enumerate}
    \item $\D$ is a Dirac structure
\begin{equation*} \label{diracstruct}
\D (x) \subset T_x\X \times T^*_x \X \times \F_R \times \E_R \times \F_P \times \E_P, \quad x \in \X,
\end{equation*}
with $\X$ the state space manifold, $\F_R$ and $\E_R=\F_R^*$ are dual flow and effort linear spaces corresponding to energy dissipation, and $\F_P$ and $\E_P=\F_P^*$ are dual flow and effort linear spaces corresponding to external ports,
\item
$\cL$ is a Lagrangian submanifold
\begin{equation*} \label{lagrangestruct}
\cL \subset T^*\X
\end{equation*}
corresponding to energy storage, and
\item
$\cR$ is a resistive relation defined by a submanifold
\begin{equation*} \label{resistiverelation}
\cR \subset \F_R \times \E_R, \quad \mbox{ such that }e_R^\top f_R \geq 0,\ \mbox{ for all}\ (f_R,e_R) \in \cR,
\end{equation*}
with $\dim \cR = \dim \F_R$.
\end{enumerate}
The resulting port-Hamiltonian DAE system is defined by the equations
\begin{equation}
\label{pHDAE}
\begin{array}{c}
(-\dot{x}(t), e(t), f_R(t),e_R(t),f_P(t),e_P(t)) \in \D (x(t)), \\[2mm]
(-f_R(t),e_R(t)) \in \cR, \quad (x(t),e(t)) \in \cL, \; t \in \mR.
\end{array}
\end{equation}

\end{definition}
Note that the description of the dynamics \eqref{pHDAE} involves the intermediate variables $f_R,e_R$ satisfying $(-f_R(t),e_R(t)) \in \cR$. In order to obtain an \emph{explicit} description in terms of only the state variables $x$ (and their time-derivatives $\dot{x}$) and the external port variables $f_P,e_P$ these variables need to be eliminated.

In order to do so, we will look closely at the \emph{composition} of the Dirac structure $\D$ and the resistive relation $\cR$. We  again first study the linear case.

%
%

\subsection{The linear case}
In this subsection we first recall from \cite{SchM23} how in the linear case the variables $f_R,e_R$ are eliminated by considering the composition of the Dirac structure $\D$ and the resistive structure $\cR$ via the resistive port variables $f_R,e_R$. For simplicity of exposition we first consider the case without external port variables $f_P,e_P$.

Throughout this subsection, we will assume that the resistive structure $\cR$ is given a by a \emph{nonnegative Lagrange structure} $\cR \subset \F_R \times \E_R$. The composition of $\cR$ with a Dirac structure $\D \subset \X \times \X^* \times \F_R \times \E_R$ via the resistive port variables $(f_R,e_R) \in \F_R \times \E_R$ is then defined as
\[
\D \circ \cR := \{(f,e) \in \X \times \X^* \mid \exists\, (-f_R,e_R) \in \cR \mbox{ s.t. } (f,e,f_R,e_R) \in \D \}.
\]
As we will demonstrate now, such a composition can be fully characterized as a \emph{maximally monotone} structure $\M \subset \X \times \X^*$. Conversely any maximally monotone structure $\M \subset \X \times \X^*$ can be written as $\M=\D \circ \cR$ for a certain choice of $\F_R,\E_R$, a Dirac structure $\D \subset \X \times \X^* \times \F_R \times \E_R$ and resistive structure $\cR \subset \F_R \times \E_R$.

For this analysis, let us first recall the definition of a maximally monotone structure.
\begin{definition}
Let $\X$ be a finite-dimensional linear space. A subspace $\M \subset \X \times \X^*$ is called a \emph{monotone structure} if the quadratic form defined by the bilinear form $\langle \cdot, \cdot \rangle_+$ is nonnegative on $\M$. It is a maximally monotone structure if moreover $\M$ is maximal with respect to this property, i.e., there does not exist another monotone structure $\M'$ with $M \subsetneqq \M'$.
\end{definition}
Note that Dirac structures are a special type of maximally monotone structures, namely maximal subspaces for which $\langle \cdot, \cdot \rangle_+$ is \emph{zero} restricted to this subspace. In general, we first recall the following proposition from \cite{SchM23}. Note that a linear map $M: \X^* \to \X$ is called {\emph monotone} if $M+M^\top \geq 0$ or, equivalently, $e^\top Me \geq 0$ for all $e \in \X^*$.
\begin{proposition}
\label{prop:monotone0}
A subspace $\M \subset \X \times \X^*$ is a monotone structure if and only if
\begin{equation*}
\label{mono}
<e \mid f>\,  \geq 0
\end{equation*}
for all $(f,e) \in \M$, and a maximally monotone structure if additionally $\M$ is maximal with respect to this property. Any monotone structure $\M$ satisfies $\dim \M \leq \dim \X$ and is a maximally monotone structure if and only if $\dim \M = \dim \X$.
Furthermore, the graph $\{(Me,e) \in \X \times \X^* \mid e \in \X^* \}$ of a linear map
$M: \X^* \to \X$ is a maximally monotone structure if and only if $M$ is monotone.
\end{proposition}
The following proposition from \cite{SchM23} formulates the equivalence between the composition of Dirac structures with resistive structures and maximally monotone structures.

\begin{proposition}
\label{prop:decompose}
Consider linear spaces $\F_R$ and $\E_R=\F_R^*$ such that $\M = \D \circ \cR$ for some  Dirac structure $\D \subset \X \times \X^* \times \F_R \times \E_R$ and a nonnegative Lagrange structure $\cR \subset \F_R \times \E_R$. Then $\M$ is a maximally monotone structure.

Conversely, let $\M \subset \X \times \X^*$ be a maximally monotone structure. Then there exist linear spaces $\F_R$ and $\E_R=\F_R^*$ such that $\M = \D \circ \cR$ for some  Dirac structure $\D \subset \X \times \X^* \times \F_R \times \E_R$ and some nonnegative Lagrange structure $\cR \subset \F_R \times \E_R$.
\end{proposition}
This means that in the definition of a linear dissipative port-Hamiltonian DAE system (with Dirac structure $\D$ and resistive structure $\cR$), we can always replace the pair $\D,\cR$ by a maximally monotone structure $\M$. Thus, dissipative port-Hamiltonian systems $(\D,\cL,\cR)$ can be equivalently defined by a pair $(\M, \cL)$, with $\M$ a maximally monotone structure, and $\cL$ a Lagrage structure.

Dissipative port-Hamiltonian systems $(\M, \cL)$ admit the following coordinate representations (not involving resistive variables $f_R,e_R$). Obviously, any subspace $ \M \subset \X \times \X^*$ can be represented in image representation as
\begin{equation*}
\M = \im \bma V \\ W \ema
\end{equation*}
for square matrices $V,W$ of dimension equal to $\dim \X$. It is easily seen that $\M$ represented in this way is maximally monotone if and only if $V,W$ are such that
\begin{equation}
\label{maxmon}
W^\top V + V^\top W \geq 0, \; \rank \bma V^\top & W^\top \ema = \dim \X.
\end{equation}
However, in order to obtain equations we need kernel representations instead of image representations. Thus let $V,W$ satisfy \eqref{maxmon}, and let $F,G$ be a maximal annihilator of $\im \bma V \\ W \ema$, i.e., $FV+GW=0$ and $\rank \bma F & G \ema= \dim \X$. Then the dissipative port-Hamiltonian system $(\M,\cL)$ takes the form
\begin{equation*}
FP\dot{z}= GSz,
\end{equation*}
where the pair $P,S$ specifies as before the Lagrange structure $\cL$. In particular, the algebraic constraints are determined by $E=FP$, and similar to the case without energy dissipation we may distinguish between algebraic constraints due to singularity of $P$ (Lagrange algebraic constraints) and due to singularity of $F$ (algebraic constraints arising from the maximally monotone structure). In particular, we have the algebraic constraints
\begin{equation*}
Sz \in \Img W,
\end{equation*}
generalizing Dirac algebraic constraints. Furthermore, it is immediately seen that the dynamics of any dissipative port-Hamiltonian systems satisfies the inequality
%
\begin{equation*}
\frac{d}{dt} \frac{1}{2} z^\top S^\top P z = z^\top S^\top P\dot{z} \leq 0,
\end{equation*}
since $\bma P\dot{z} \\ -Sz \ema \in \M = \ker \bma F & G \ema$.

While \eqref{maxmon} is a direct characterization of maximal monotonicity of a subspace $\M$ given in image representation, we do not yet have a corresponding characterization for kernel representations $F,G$. This is done in the following proposition.
\begin{proposition}\label{prop:maxmon}
The matrix
\[
\M = \im \bma V \\ W \ema
\]
is the image representation of a maximally monotone structure if and only if for every  maximal annihilator $
\mat{cc} F & G\rix$ of $\M$ one has $GF^\top+F G^\top \leq 0$.
\end{proposition}
\proof
%
%
%
%
%
%
If $\M= \ker \bma F & G \ema$ with $F$ \emph{invertible}, then the $(f,e) \in \M$ are satisfying $f= - F^{-1}Ge$. Since $<e \mid f> \, \geq \, 0$ for all $(f,e) \in \M$, it follows that $e^\top F^{-1}Ge \leq 0$ for all $e$. Thus $F^{-1}G + G^\top F^{-\top} \leq 0$, or equivalently $GF^\top + FG^\top \leq 0$.
On the other hand, if $\M$ is such that $F$ is \emph{not} invertible, then arbitrarily close to $\M$ there will exist a maximally monotone $\M'$ with associated $F'$ being invertible, and thus by continuity $GF^\top + FG^\top \leq 0$ remains to hold.
\eproof

These observations are  immediately extended to linear dissipative port-Hamiltonian systems with \emph{external ports}. In fact, let $\F_P \times \E_P$, with $\E_P=\F_P^*$ (here the index $P$ denotes external ports), be the space of external flows and efforts (inputs and outputs). Then a linear dissipative port-Hamiltonian system with external ports is specified by a maximally monotone structure
$\M \subset \X \times \X^* \times \F_P \times \E_P$,
together with a Lagrange structure $\cL \subset \X \times \X^*$ modeling energy storage.

After recalling the extension of linear port-Hamiltonian DAE systems to system with resistive and external ports, in the next subsection we proceed analogously for the nonlinear case.

\subsection{The nonlinear case}

In order to eliminate the resistive port variables $f_R,e_R$ in the nonlinear case we proceed as follows. Importantly, here we \emph{restrict} ourselves to  \emph{modulated linear} energy dissipation relations $\cR$ which means that $\cR$ is assumed to be defined by subspaces
\begin{equation}
\cR(x) \subset \F_R \times \E_R,
\end{equation}
depending on $x \in \X$, in the sense that for every $x \in \X$ the subspace $\cR(x) \subset \F_R \times \E_R$ is a nonnegative Lagrange structure. Thus, energy dissipation is assumed to be a \emph{modulated} resistive structure. In this case we can apply the same reasoning as in the linear case in the following way.

Let $\D$ be a Dirac structure on $\X$ as before. Thus, for every $x \in \X$ we have that
\begin{equation}
\D(x) \subset T_x\X \times T_x^*\X \times \F_R \times \E_R.
\end{equation}
is a Dirac structure as in the linear case. Then consider
for every $x \in \X$ the composition
\begin{equation}
\begin{array}{l}
\M(x):=\D(x) \circ \cR(x) = \{(f,e) \in T_x\X \times T_x^*\X \mid \\[2mm] \qquad
\exists\, (-f_R,e_R) \in \cR(x) \mbox{ s.t. } (f,e,f_R,e_R) \in \D(x) \}.
\end{array}
\end{equation}
Using the same arguments as in the linear case, this defines for every $x$ a maximally monotone structure $\M(x) \subset T_x\X \times T_x^*\X$. Thus the composition of the Dirac structure modulated by $x$ and the resistive structure modulated by the same $x$ defines a maximally monotone structure, again modulated by this $x$.
Conversely, again using the results of the linear case, under regularity conditions every maximally monotone structure $\M(x)$ modulated by $x$ can be written as the composition of a Dirac structure $\D(x)$ and resistive structure $\cR(x)$, both modulated by $x$. Using the matrix representation developed for the linear case this means that there exist matrices $F(x),G(x)$ such that
\begin{equation}
\M(x) =\{ (f,e) \in T_x\X \times T_x^*\X \mid F(x)f+ G(x)e=0 \},
\end{equation}
where
\begin{equation}
F(x)G^\top(x) + G(x)F^\top(x) \leq 0, \quad \rank \bma F(x) & G(x) \ema = \dim \X, x
\in \X.
\end{equation}
If the Lagrangian submanifold is generated by a function $\mathcal H$, this means that the dissipative port-Hamiltonian dynamics is given as
\begin{equation}
F(x) \dot{x}=G(x) \frac{\partial \mathcal H}{\partial x}(x).
\end{equation}
In particular, if $F(x)$ is invertible then the dynamics takes the form
$\dot{x}=-M(x) \frac{\partial \mathcal H}{\partial x}(x)$ for some matrix (modulated by $x$) that is monotone, i.e., $M(x) + M(x)^\top \geq 0$. Splitting $M(x)=J(x)-R(x)$ into its skew-symmetric and symmetric part, the dynamics thus takes the well-known port-Hamiltonian ODE form
\begin{equation}
\dot{x} = \left[J(x) -R(x)\right] \frac{\partial \mathcal H}{\partial x}(x),\quad  J(x)=-J^\top (x), R(x)=R^\top (x) \geq 0.
\end{equation}
Finally, external port variables $f_P,e_P$, i.e., inputs and outputs, are added by starting from an extended modulated Dirac structure
\begin{equation}
\D(x) \subset T_X\X \times T_x^*\X \times \F_R \times \E_R \times \F_P \times \E_P,
\end{equation}
together with a modulated resistive structure $\cR(x)$. Composition through the $f_R,e_R$ variables leads to a modulated maximally monotone structure
\begin{equation}
\M(x) \subset T_x\X \times T_x^*\X \times \F_P \times \E_P,
\end{equation}
given as $\M(x) = \ker \bma F(x) & G(x) & F_P(x) & G_P(x) \ema$ for matrix functions $F(x)$, $G(x)$, $F_R(x)$, $G_R(x)$ satisfying
\begin{equation}
\begin{array}{l}
F(x)G^\top(x) + G(x) F^\top(x) + F_P(x)G_P^\top(x) + G_P(x) F_P^\top(x) \leq 0, \\[2mm]
\rank \bma F(x) & G(x) & F_P(x) & G_P(x) \ema = \dim \X \times \F_P.
\end{array}
\end{equation}
The resulting dissipative port-Hamiltonian system with external ports is thus given as
\begin{equation}
F(x)\dot{x}=G(x) \frac{\partial H}{\partial x}(x) + F_P(x)f_P + G_P(x)e_P.
\end{equation}
with external port variables $f_P,e_P$.

In this section we have shown how to incorporate dissipation as well as external ports into port-Hamiltonian differential equations which completes the picture.

\section{Conclusion and Outlook}\label{sec:conclusion}

We have presented a systematic study of constrained nonlinear port-Ha\-mil\-to\-ni\-an systems and have shown how constraints (of Lagrange or Dirac type) can be incorporated into the port-Hamiltonian formulation, and how one can be converted into the other. The analysis uses Morse families and their skew counterparts and is illustrated in detail for mechanical systems. It is also shown how to incorporate dissipation and external ports. Further extensions could be the analysis of the regularity requirements (index) of these equations, as well as the extension to infinite dimensional constrained port-Hamiltonian operator equations.



\section{Appendix A: Augmentation of Dirac, Lagrange, and  maximally monotone structures to graphs of mappings.}
We recall from \cite{SchM18}, see also \cite{DalS99}, the following basic result concerning Dirac structures on linear spaces. Denote the projection from $\X \times \X^*$ to $\X$ by $\pi$, and from $\X \times \X^*$ to $\X^*$ by $\pi^*$.
\begin{proposition}
\label{prop:diracaug}
Let $\X$ be a finite-dimensional linear space. Consider any Dirac structure $\D \subset \X \times \X^*$, and suppose that $\pi^*(\D) \subset \X^*$ is $(n-k)$-dimensional, where $\pi^*: \X \times \X^* \to \X^*$ is the projection onto the second argument. With $\Lambda:= \mR^k$,  there exists a full-rank $n \times k$ map $C_\lambda :\Lambda^* \to \X$ and a skew-symmetric map $J: \X^* \to \X $, such that $\D$ is given as
\begin{equation*}
\D = \{(f,e) \in \X \times \X^* \mid\ \exists\ e_\lambda \in \Lambda^* \mbox{ s.t. } -f = Je + C_\lambda e_\lambda,\, 0 = C_\lambda^\top e \}.
\end{equation*}
In fact, one can take any $C_\lambda$ such that $\ker C_\lambda^\top = \pi^* (\D)$. Conversely, any such equations for a skew-symmetric map $J: \X^* \to \X$ define a Dirac structure.
\end{proposition}
Hence any Dirac structure $\D \subset \X \times \X^*$ can be {\it augmented} to the graph of the skew-symmetric map
\begin{equation*}
J_a:=
\bma - J & - C_\lambda \\[2mm]
C_\lambda^\top & 0 \ema
\end{equation*}
from $\X^* \times \Lambda^*$ to $\X \times \Lambda$.

Analogously, for Lagrange structures one has the following result, see \cite{SchM18}.
\begin{proposition}
\label{prop:lagrangeaug}
For any Lagrange structure $\cL\in  \X \times \X^*$ there exists a full-rank $n \times k$ matrix $C_\mu$ and a symmetric $n \times n$ matrix $Q$ such that
\begin{equation*}
\cL = \{(x,e) \in \X \times \X^* \mid\ \exists \mu \in \mR^k \mbox{ s.t. }e =  Qx + C_\mu \mu, 0  =  C_\mu^\top x \}.
\end{equation*}
In fact, one can take any $C_\mu$ such that $\ker C_\mu^\top = \pi (\cL)$. Conversely, any such equations for a symmetric map $Q: \X \to \X^*$ define a Lagrange structure.
\end{proposition}
Hence any Lagrange structure $\cL \subset \X \times \X^*$ can be {\it augmented} to the graph of a symmetric map
\begin{equation*}
Q_a:=\begin{bmatrix} Q & C_\mu\\[2mm] C_\mu^\top & 0 \end{bmatrix} : \X \times \mR^k \to \X^* \times \mR^{k*}
\end{equation*}
The quadratic function
$\frac{1}{2}x^\top Qx + \mu^\top C_\mu^\top x$ associated with this symmetric map is called a \emph{Morse family} of the Lagrange structure $\cL$.

Finally, as shown in \cite{SchM23}, any maximal monotone structure $\M \subset \X \times \X^*$ can be augmented to the graph of a monotone mapping. Recall that a mapping $M:\X^* \to \X$ is \emph{monotone} if its matrix representation (also denoted by $M$) satisfies $M + M^\top \geq 0$.
\begin{proposition}
\label{prop:monotoneaug}
Consider a maximally monotone structure $\M \subset \X \times \X^*$. Write $\pi^* \M= \ker C_\lambda^\top$ for some matrix $C_\lambda: \Lambda \to \X$ with full column rank and linear space $\Lambda$. Then there exists a monotone map $M: \X^* \to \X$ such that
\begin{equation*}
\label{eq:monrep}
\M = \left \{(f,e) \in \X \times \X^* \mid\ \exists\ e_\lambda \in \Lambda^* \mbox{ s.t. } \bma f \\[2mm] 0 \ema = \bma -M & -C_\lambda \\[2mm] C_\lambda ^\top  & 0 \ema \bma e \\[2mm] e_\lambda \ema \right \}.
\end{equation*}
Conversely, for any monotone map $M$ this defines a maximally monotone structure.
\end{proposition}
Thus, any maximally monotone structure $\M$ can be augmented to the graph of the monotone mapping
\begin{equation*}
M_a:= \bma -M & -C_\lambda \\ C_\lambda^\top  & 0 \ema : \X^* \times \Lambda^* \to \X \times \Lambda.
\end{equation*}

\section{Appendix B: Augmentation of Dirac structures and Lagrangian submanifolds; Morse families.}

As a direct extension of the linear case (Proposition \ref{prop:diracaug}), any Dirac structure $\D$ on a manifold $\X$ can be locally written as the graph of a modulated skew-symmetric map on an augmented state manifold $\X_a$, cf. \cite{DalS99}. Denote in analogy with the linear case the projection of $T_x\X \times T_x^*\X$ to $T_x\X$ by $\pi$, and the projection of $T_x\X \times T_x^*\X$ to $T_x^*\X$ by $\pi^*$.
\begin{proposition}
\label{prop:nonldiracaug}
Consider any Dirac structure $\D$ on a manifold $\X$. Suppose that $\pi^*(\D(x)) \subset T_x\X^*$ is $(n-k)$-dimensional and set $\Lambda:= \mR^k$. Then there exists a full-rank $n \times k$ matrix function $C_\lambda(x): \Lambda^* \to T_x\X$ and a skew-symmetric $n \times n$ matrix function $J(x)$ such that locally
\begin{equation}
\label{Diracconstraint2}
\begin{array} {l} \D(x) = \left \{(f,e) \! \in T_x\X \times T_x^*\X \mid \exists\, e_\lambda \in \Lambda^* \mbox{ s.t. }-f = J(x)e + C_\lambda(x)e_\lambda,\right .\\[2mm] \qquad \qquad\left . 0  =  C_\lambda^\top(x)e \right \}.
\end{array}
\end{equation}
Conversely, any such equations for a skew-symmetric map $J(x): T_x^*\X \to T_x\X$ define a Dirac structure.
\end{proposition}

The generalization of Proposition \ref{prop:lagrangeaug} from Lagrange structures to Lagrangian manifolds is more involved. See e.g. \cite{BarPM15} and the references quoted in there.
\begin{proposition}
\label{prop:nonllagrangeaug}
Consider a Lagrangian submanifold $\cL \subset T^*\X$, a point $\ell \in \cL$ and its projection $\pi(\ell) \in \X$. Then there exists a neighborhood $V$ of $\pi(\ell)$, a natural number $k$, a neighborhood $W$ of $0 \in \mR^k$, together with a smooth function $F: V \times W \to \mR$, such that the rank of $\frac{\partial F}{\partial \mu}$, with $\mu \in \mR^k$, is equal to $k$ on $\left(\frac{\partial F}{\partial \mu}\right)^{-1}(0)$, and
\begin{equation}
\label{morse3}
\left \{ \left(x, \frac{\partial F}{\partial x}(x, \mu) \right) \mid \frac{\partial F}{\partial \mu}(x, \mu)=0 \right \}
\end{equation}
is a neighborhood of the point $\ell$ in $\cL$. The function $F(x,\mu)$, seen as a function of $x$ and parametrized by $\mu \in \mR^k$, is called a Morse family for the Lagrangian submanifold $\cL$.
\end{proposition}
Thus, locally any Lagrangian submanifold $\cL \subset T^*\X$ can be locally augmented to a \emph{regular} Lagrangian submanifold $\cL_a \subset T^* (\X \times \mR^k)$, with generating function $F(x,\mu)$.
\end{document}